\documentclass[a4paper,12pt]{article}
\title{Comparison results of $P_2$-finite elements for fourth-order semilinear von K\'{a}rm\'{a}n equations}
\author {Gouranga Mallik\footnote{Department of Mathematics, Indian Institute of Science, Bangalore - 560012, India. Email. gourangam@iisc.ac.in}
}
\usepackage{amsmath,amsthm,amssymb,enumerate}
\usepackage{gensymb}
\usepackage[square,sort&compress,comma,numbers]{natbib}

\usepackage[sc]{mathpazo}
\usepackage{multirow} 

\usepackage{multicol}
\usepackage{newtxtext,newtxmath}
\usepackage{subfig}
\usepackage{graphicx}
\usepackage{epstopdf}
\usepackage{hyperref}
\usepackage{cancel}
\usepackage[margin=3cm]{geometry}
\usepackage{tikz}
\usepackage{caption}
\usetikzlibrary{shapes,calc}
\usepackage{verbatim}
\usepackage{mathrsfs}
\usepackage{algorithm}
\usepackage{accents}

\chardef\bslash=`\\ 

\newtheorem{thm}{Theorem}[section]
\newtheorem{cor}[thm]{Corollary}
\newtheorem{lem}[thm]{Lemma}

\theoremstyle{definition}

\theoremstyle{remark}
\newtheorem{rem}{Remark}[section]

\numberwithin{equation}{section}

\newcommand{\bR}{\mathbb R}

\newcommand{\cE}{\mathcal E}

\newcommand{\cT}{\mathcal T}

\newcommand{\cN}{\mathcal N}

\newcommand{\bX}{\boldsymbol{X}}

\newcommand{\bv}{\boldsymbol{v}}

\newcommand{\map}{\longrightarrow}

\newcommand{\lt}{L^2(\Omega)}

\newcommand{\hto}{H^2_0(\Omega)}

\newcommand{\St}{P_2(\cT)}

\newcommand{\vket}{von K\'{a}rm\'{a}n equations }

\newcommand{\integ}{\int_\Omega}

\newcommand{\sik}{\sum_{K\in\mathcal{T}}\int_K}

\newcommand{\se}{\sum_{E\in \mathcal{E}}\int_E}
\newcommand{\sie}{\sum_{E\in \mathcal{E}(\Omega)}\int_E}

\newcommand{\fl}{\quad\forall}

\newcommand{\half}{\frac{1}{2}}
\newcommand{\trinl}{\ensuremath{|\!|\!|}}
\newcommand{\trinr}{\ensuremath{|\!|\!|}}

\newcommand{\htk}{H^2(K)}

\newcommand{\dx}{{\rm\,dx}}

\newcommand{\ds}{{\rm\,ds}}

\newcommand{\MS}{{\rm M}(\cT)}

\newcommand*{\avintK}{\mathop{\ooalign{$\int_K$\cr$-$}}}
\newcommand*{\avintE}{\mathop{\ooalign{$\int_E$\cr$-$}}}

\newcommand{\bMS}{{\boldsymbol{\rm M}}(\cT)}
\newcommand{\bIPS}{{\boldsymbol{\rm IP}}(\cT)}
\newcommand{\IPS}{{\rm IP}(\cT)}
\newcommand{\DGS}{{P}_2(\cT)}
\newcommand{\bDGS}{\boldsymbol{\rm P}_2(\cT)}

\newcommand{\M}{{\rm M}}
\newcommand{\nc}{{\rm NC}}
\newcommand{\dg}{{\rm dG}}
\newcommand{\ip}{{\rm IP}}

\newcommand{\Holder}{H\"{o}lder }

\newcommand{\jump}[1]{\left[#1\right]_E}
\newcommand{\avg}[1]{\left\langle#1\right\rangle_E}
\begin{document}
\date{\today}
\maketitle

\begin{abstract} 
Lower-order $P_2$ finite elements are popular for solving fourth-order elliptic PDEs when the solution has limited regularity.  
{\it A priori} and {\it a posteriori} error estimates for von K\'{a}rm\'{a}n equations are considered in Carstensen {\it et al.} \cite{CCGMNN_DG,CCGMNN_Semilinear} with respect to different mesh dependent norms which involve different jump and penalization terms. This paper addresses the question, whether they are comparable with respect to a common norm. This article  establishes that the errors for the quadratic symmetric interior discontinuous Galerkin, $C^0$ interior penalty and nonconforming Morley finite element methods are equivalent upto some higher-order oscillation term with respect to a unified norm. Numerical experiments are performed to substantiate the comparison results.
\end{abstract}


\section{Introduction}
This paper concerns the comparison results of $P_2$-finite element approximations of regular solution to the von K\'{a}rm\'{a}n equations defined on $\Omega\subset\bR^2$, which describe the deflection of very thin elastic plates. Those plates are  modeled by a semi-linear system of fourth-order coupled partial differential equations (PDEs) and can be described as follows. For a given load function $f\in\lt$, seek vertical displacement $u$ and Airy's stress $v$ 
such that
\begin{subequations}\label{vke}
\begin{align}
&\Delta^2 u =[u,v]+f &&
\text{ in } \Omega,\\
&\Delta^2 v =-\half[u,u] &&
\text{ in } \Omega,\\
&u=\frac{\partial u}{\partial \nu} = v = \frac{\partial v}{\partial \nu} = 0 &&\text{  on  } \partial\Omega,
\end{align}	
\end{subequations}
\noindent with the biharmonic operator $\Delta^2$ and the von K\'{a}rm\'{a}n bracket $[\bullet,\bullet]$ are defined as
$\displaystyle\Delta^2\varphi:=\varphi_{xxxx}+2\varphi_{xxyy}+\varphi_{yyyy}$, and $\displaystyle
[\eta,\chi]:=\eta_{xx}\chi_{yy}+\eta_{yy}\chi_{xx}-2\eta_{xy}\chi_{xy}$.

The quasi-optimality results of Gudi \cite{Gudi10} on medius analysis for linear biharmonic problem imply that the errors of these methods are comparable with best-approximation in the finite element space. The comparisons are made with respect to different discrete norms which depend on the underlying finite element spaces. Carstensen et al. \cite{CC_DG_NN_15_Comparison} extends this results to $P_2$ finite elements with an equivalent unified norm.

Though there are many research on medius analysis for linear PDEs \cite{Gudi10,CC_DG_NN_15_Comparison}, but there are very few results \cite{CCGMNN_Semilinear} for nonlinear PDEs. This paper establishes a comparison result of $P_2$ finite elements for semilinear \vket with respect  to a unified norm $\trinl\bullet\trinr_h$ as:
\begin{equation*}
\trinl \Psi-\Psi_{\M}\trinr_h\approx\trinl \Psi-\Psi_{\ip}\trinr_h\approx\trinl \Psi-\Psi_{\dg}\trinr_h
\end{equation*}
upto some oscillation, where $\Psi_{\M},\Psi_{\ip}$ and $\Psi_{\dg}$ are the approximate solutions to \eqref{vke} for nonconforming, $C^0$ interior penalty and discontinuous Galerkin finite element methods respectively.
The optimal convergence rates are achieved in numerical experiments, when meshes are adapted by a posteriori estimators.

\medskip
Throughout the paper, standard notation on Lebesgue and Sobolev spaces and their norms are employed.
The standard semi-norm and norm on $H^{s}(\Omega)$ (resp. $W^{s,p} (\Omega)$) for $s>0$ are denoted by $|\bullet|_{s}$ and $\|\bullet\|_{s}$ (resp. $|\bullet|_{s,p}$ and $\|\bullet\|_{s,p}$ ). The duality pairing between $X$ and its dual space $X^*$ is denoted by $(\bullet,\bullet)$. Bold letters, e.g. $\bX= X\times X$ refer to product spaces and Greek letters refer to vector valued functions. The positive constants $C$ appearing in the inequalities denote generic constants which do not depend on the mesh-size. The notation $a\lesssim b$ means that there exists a generic constant $C$ independent of the mesh parameters and independent of the stabilization parameters $\sigma_1$ and $\sigma_2\geq 1$ such that $a \leq Cb$; $a\approx b$ abbreviates $a\lesssim b\lesssim a$.

\section{Preliminaries}
This section introduces weak formulation for the von K\'{a}rm\'{a}n equations and states some known results.
The weak formulation of von K\'{a}rm\'{a}n equations \eqref{vke} reads: Given $f\in\lt$, seek  $u,v\in \: X:=\hto$ such that
\begin{subequations}\label{wform}
	\begin{align}
	& a(u,\varphi_1)+ b(u,v,\varphi_1)+b(v,u,\varphi_1)=l(\varphi_1)   \fl\varphi_1\in X\label{wforma}\\
	& a(v,\varphi_2)-b(u,u,\varphi_2)   =0            \fl\varphi_2 \in X,\label{wformb}
	\end{align}
\end{subequations}
where, for all $\eta,\chi,\varphi\in X$, 
\begin{align}
&a(\eta,\chi):=\integ D^2 \eta:D^2\chi\dx,\; \; b(\eta,\chi,\varphi):=-\half\integ [\eta,\chi]\varphi\dx, \text{ and } l(\varphi):=\int_{\Omega}f\varphi\dx. \label{defnab}
\end{align}
Given $F=(f,0)\in L^2(\Omega)\times L^2(\Omega)$, the combined vector form seeks $\Psi=(u,v)\in \bX:=X\times X\equiv\hto\times\hto$ such that
\begin{equation}\label{vform_cts}
N(\Psi,\Phi)=(N(\Psi),\Phi):=A(\Psi,\Phi)+B(\Psi,\Psi,\Phi)-L(\Phi)=0\fl \Phi\in \bX,
\end{equation}
where, for all $\Xi=(\xi_1,\xi_2),\Theta=(\theta_1,\theta_2)$, and $\Phi=(\varphi_1,\varphi_2)\in  \bX$,
\begin{align*}
& A(\Theta,\Phi):=a(\theta_1,\varphi_1)+a(\theta_2,\varphi_2),\\
&B(\Xi,\Theta,\Phi):=b(\xi_1,\theta_2,\varphi_1)+b(\xi_2,\theta_1,\varphi_1)-b(\xi_1,\theta_1,\varphi_2)\text{ and } L(\Phi):=l(\varphi_1).
\end{align*}
Since $b(\bullet,\bullet,\bullet)$ is symmetric in first two variables, the trilinear form $B(\bullet,\bullet,\bullet)$ is symmetric in first two variables.

Let $\trinl\bullet\trinr_2$ denote the product norm on $\bX$  defined by $\trinl\Phi\trinr_2:=\left(|\varphi_1|_{2}^2+|\varphi_2|_{2}^2\right)^{1/2}$ for all $\Phi=(\varphi_1,\varphi_2)\in \bX$. It is easy to verify the boundedness and ellipticity properties
\begin{align*}
&{A}(\Theta,\Phi)\leq \trinl\Theta\trinr_2 \: \trinl\Phi\trinr_2,\: {A}(\Theta,\Theta) \geq \trinl\Theta\trinr_2^2,\\
&\quad B(\Xi, \Theta, \Phi) \leq  C \trinl\Xi\trinr_2 \: \trinl\Theta\trinr_2 \: \trinl\Phi\trinr_2.
\end{align*}

\medskip
For the existence of solution to \eqref{vform_cts}, regularity and bifurcation phenomena, we refer  to \cite{CiarletPlates, Knightly, BergerFife66, Berger,BergerFife, BlumRannacher}. For given $f\in H^{-1}(\Omega)$, it is well known \cite{BlumRannacher} that on a polygonal domain $\Omega$, the solutions $u,v$ belong to $\hto\cap H^{2+\alpha}(\Omega)$, where the index of elliptic regularity $\alpha\in (\half,1]$ determined by the interior angles of $\Omega$. Note that when $\Omega$ is convex; $\alpha=1$; that is, the solution belongs to $\hto\cap H^3(\Omega) $. Unless specified otherwise, the parameter $\alpha$ is supposed to satisfy $1/2<\alpha\leq 1$. 

\medskip

Denote the  Gateaux derivative of $N(\Psi)$ at $\Psi$ in the direction $\Theta$ by $DN(\Psi;\Theta)$. Due to symmetry of $B(\bullet,\bullet,\bullet)$, we have $DN(\Psi;\Theta,\Phi)=(DN(\Psi;\Theta),\Phi)=A(\Theta,\Phi)+2B(\Psi,\Theta,\Phi)$.
Throughout the paper, we consider the approximation of a regular solution \cite{Brezzi,GMNN_BFS} $\Psi$ to the non-linear map $N(\Psi)=0$ of \eqref{vform_cts} in the sense that the bounded derivative $DN(\Psi;\Theta,\Phi)$ satisfies the inf-sup condition 
\begin{align}\label{inf_sup_cts}
0<\beta:=\inf_{\substack{\Theta\in \bX\\ \trinl\Theta\trinr_2=1}}\sup_{\substack{\Phi\in \bX\\ \trinl\Phi\trinr_2=1}}DN(\Psi;\Theta,\Phi).
\end{align}

\section{Finite element methods and their comparison}
Let  $\cT$ be a shape-regular \cite{Braess} triangulation of the bounded polygonal Lipschitz domain $\Omega\subset\bR^2$ into closed triangles. 
The set of all internal vertices (resp. boundary vertices) and  interior edges (resp.  boundary edges)  of the triangulation $\cT$ are denoted by $\cN (\Omega)$ (resp.  $\cN(\partial\Omega)$) and $\cE (\Omega)$ (resp. $\cE (\partial\Omega)$).
Define a piecewise constant mesh function $h_{\cT}(x)=h_K={\rm diam} (K)$ for all $x \in K$, $ K\in \cT$, and set $h:=\max_{K\in \cT}h_K$. Also define a piecewise constant edge-function on $\cE:=\cE(\Omega)\cup \cE(\partial\Omega)$ by $h_{\cE}|_E=h_E={\rm diam}(E)$ for any $E\in \cE$. Set of all edges of $K$ is denoted by $\cE(K)$. Note that for a shape-regular family, there exists a positive constant $C$ independent of $h$ such that any $K\in\cT$ and any $E\in \partial K$ satisfy 
\begin{equation}\label{shape_reg_const}
Ch_K\leq h_E\leq h_K.
\end{equation}
\noindent Let $P_r( K)$ denote the set of all polynomials of degree less than or equal to $r$ and $\displaystyle P_r(\cT):=\left\{\varphi\in L^2(\Omega):\,\forall K\in\cT,\varphi|_{K}\in P_r(K)\right\}$ and write $\boldsymbol{\rm P}_r(\cT):=P_r(\cT)\times P_r(\cT)$ for pairs of piecewise polynomials.
For a nonnegative integer $s$, define the broken Sobolev space for the subdivision $\cT$  as
\begin{equation*}
H^s(\cT)=\left\{\varphi\in\lt: \varphi|_K\in H^{s}(K)\fl K\in \cT \right\}
\end{equation*}
with the broken Sobolev semi-norm  $|\bullet|_{H^s(\cT)}$ and norm $\| \bullet\|_{H^s(\cT)}$ defined by
\begin{equation*}
|\varphi|_{H^s(\cT)}=\bigg{(}\sum_{K\in\cT} |\varphi|_{H^{s}(K)}^2\bigg{)}^{1/2}\text{ and }
\|\varphi\|_{H^s(\cT)}=\bigg{(}\sum_{K\in\cT}\|\varphi\|_{H^{s}(K)}^2\bigg{)}^{1/2}.
\end{equation*}
Define the jump $[\varphi]_E=\varphi|_{K_+}-\varphi|_{K_-}$ and the average $\langle\varphi\rangle_E=\half\left(\varphi|_{K_+}+\varphi|_{K_-}\right)$ across the interior edge $E$ of $\varphi\in H^1(\cT)$ of the adjacent triangles  $K_+$ and $K_-$. Extend the definition of the jump and the average to an edge lying in boundary by $[\varphi]_E=\varphi|_E$ and $\langle\varphi\rangle_E=\varphi|_E$ { when $E$ belongs to the set of boundary edges $\cE(\partial\Omega)$}.
For any vector function, jump and average are understood componentwise.
The union of all edges reads $\Gamma\equiv\bigcup_{E\in\cE}E$. 
Define a general nonconforming norm $\displaystyle\|v_{h}\|_{\nc}^2:=\sik |D^2v_h|^2\dx$ for $v_h\in H^2(\cT)+P_2(\cT)$. 

\subsection{Morley finite element}
The nonconforming Morley element space $\MS$ associated with the triangulation $\cT$ is defined by
\begin{equation*}
\MS:=\Bigg{\{} v_M\in P_2(\cT){{\Bigg |}}
\begin{aligned}
& v_M \text{ is continuous at } \cN(\Omega) \text{ and vanishes at }  \cN(\partial \Omega), \:   {\text{for all}} \:   \\
&E \in \cE (\Omega) \:\int_{E}\left[\frac{\partial v_M}{\partial \nu}\right]_E\ds=0;\:  {\text{for all}} \: E\in \cE (\partial\Omega) \: \int_{E}\frac{\partial v_M}{\partial \nu}\ds=0
\end{aligned}
\Bigg{\}}.
\end{equation*}
Define the discrete bilinear, trilinear and linear forms by
\begin{align*}
a_{\nc}(\eta,\chi)&:=\sik D^2 \eta:D^2\chi\dx\\ b_{h}(\eta,\chi,\varphi)&:=-\half\sik [\eta,\chi]\varphi\dx, \text{ and } l_{h}(\varphi):=\sik f\varphi\dx.
\end{align*}
A nonconforming finite element formulation corresponding to \eqref{vform_cts} seeks $\Psi_{\M}\in \bMS:=\MS\times \MS$ such that
\begin{equation}\label{vformd_nc}
N_{\nc}(\Psi_{\M};\Phi_{\M}):=A_{\nc}(\Psi_{\M},\Phi_{\M})+B_{h}(\Psi_{\M},\Psi_{\M},\Phi_{\M})-L_{h}(\Phi_{\M})=0 \fl \Phi_{\M} \in  \bMS,
\end{equation}
where the vector discrete bilinear, trilinear and linear forms read: for all $ \Xi=(\xi_{1},\xi_{2}),\Theta=(\theta_{1},\theta_{2})$ and $ \Phi=(\varphi_{1},\varphi_{2})\in \bMS$,
\begin{align}
&A_{\nc}(\Theta,\Phi):=a_{\nc}(\theta_1,\varphi_1)+a_{\nc}(\theta_2,\varphi_2), \notag\\
&B_{h}(\Xi,\Theta,\Phi):=b_{h}(\xi_{1},\theta_{2},\varphi_{1})+b_{h}(\xi_{2},\theta_{1},\varphi_{1})-b_{h}(\xi_{1},\theta_{1},\varphi_{2}),\label{defn_Bh}\\
&L_{h}(\Phi):=l_{h}(\varphi_1).\label{defn_Lh}
\end{align} 
The existence, local uniqueness and error estimates for the discrete solution of \eqref{vformd_nc} are shown in \cite{CCGMNN_Semilinear} for sufficiently small mesh parameter $h$. In the next lemma, an interpolation result  is defined and its results are stated.

\begin{lem}[Morley interpolation] \label{Morley_Interpolation} \cite{HuShi_Morley_Apost,CCDGJH14}
	For any $v\in X+ \MS(\cT)$, the Morley interpolation 
	$I_M(v)\in \MS$  defined by
	\begin{equation*}
	(I_M v)(z)=v(z) \text{ for any } z\in \cN(\Omega) \text{ and } 
	\int_E\frac{\partial I_M v}{\partial \nu_E}\ds=\int_E\frac{\partial v}{\partial \nu_E}\ds \text{ for any } E\in \cE
	\end{equation*}
	satisfies the integral mean property of the Hessian \\
	(a) $D^2_{\text{\rm pw}} I_M =\Pi_0 D^2$ and,	
	\begin{equation*}
	(b)\: \|h_K^{-2}(1-I_M)v \|_{L^2(K)}+\|h_K^{-1}\nabla(1-I_M) v\|_{L^2(K)}
	+{\|D^2 I_Mv\|_{L^2(K)}}\lesssim
	\|D^2v\|_{L^2(K)}.\qquad\qed
	\end{equation*}
\end{lem}


\subsection{$C^0$ finite element}
The $C^0$ IP method is based on the continuous Lagrange $P_2$ finite element space
\begin{equation*}
\ip(\cT):=P_2(\cT)\cap H^1_0(\Omega).
\end{equation*}
Define the discrete bilinear, trilinear and linear forms by: for all $\eta_{\ip},\chi_{\ip}$ and $\varphi_{\ip}\in\ip(\cT)$
\begin{align}
&a_{\ip}(\eta_{\ip},\chi_{\ip}):=a_{\nc}(\eta_{\ip},\chi_{\ip})+\se \left\langle \frac{D^2\eta_{\ip}}{\partial\nu_E^2}\right\rangle_E \left[\frac{\partial\chi_{\ip}}{\partial\nu_E}\right]_E\ds\notag\\
&\qquad+\se \left\langle \frac{D^2\chi_{\ip}}{\partial\nu_E^2}\right\rangle_E \left[\frac{\partial\eta_{\ip}}{\partial\nu_E}\right]_E\ds+\sum_{E\in\cE}\frac{\sigma_{\ip}}{h_E}\int_E\left[\frac{\partial\eta_{\ip}}{\partial\nu_E}\right]_E\left[\frac{\partial\chi_{\ip}}{\partial\nu_E}\right]_E\ds.
\end{align}
The $C^0$IP norm on $\ip(\cT)$ is defined by $$\|\eta_{\ip}\|_{\ip}^2:=\|\eta_{\ip}\|_{\nc}^2+\sum_{E\in\cE}h_E^{-1}\left\|\left[\frac{\partial\eta_{\ip}}{\partial\nu_E}\right]_E\right\|_{L^2(E)}^2.$$
For sufficiently large penalty parameter $\sigma_{\ip}$, the coercivity result \cite{BGS10} $\|\bullet\|_{\ip}^2\lesssim a_{\ip}(\bullet,\bullet)$ on $\ip(\cT)$ holds.
A $C^0$ IP finite element formulation corresponding to \eqref{vform_cts} seeks $\Psi_{\ip}\in \bIPS:=\IPS\times \IPS$ such that
\begin{equation}\label{vformd_ip}
N_\ip(\Psi_{\ip};\Phi_{\ip}):=A_{\ip}(\Psi_{\ip},\Phi_{\ip})+B_{h}(\Psi_{\ip},\Psi_{\ip},\Phi_{\ip})-L_{h}(\Phi_{\ip})=0 \fl \Phi_{\ip} \in  \bIPS,
\end{equation}
where for all $ \Xi=(\xi_{1},\xi_{2}),\Theta=(\theta_{1},\theta_{2})$ and $ \Phi=(\varphi_{1},\varphi_{2})\in\bIPS$, the bilinear form $A_{\ip}(\Theta,\Phi):=a_{\ip}(\theta_1,\varphi_1)+a_{\ip}(\theta_2,\varphi_2)$, trilinear and linear forms
$B_h(\Xi,\Theta,\Phi)$ and $L_h(\Phi)$ are respectively as defined in \eqref{defn_Bh} and \eqref{defn_Lh}.
The existence, local uniqueness and error estimates for the discrete solution of \eqref{vformd_ip} are shown in \cite{CCGMNN_DG,BS_C0IP_VKE} for sufficiently small mesh parameter $h$.

\subsection{Discontinuous Galerkin finite element}
In this section, a discontinuous Galerkin method of \cite{Baker_77,CC_DG_NN_15_Comparison,Feng_Karakashian_07} for the biharmonic part is proposed. Define the bilinear, trilinear and linear forms by, for $\eta_{\dg},\chi_{\dg}$ and $\varphi_{\dg}\in P_2(\cT)$ and the penalty parameter $\sigma_{\dg}>0$
\begin{align}
&a_{\dg}(\eta_{\dg},\chi_{\dg}):=a_{\nc}(\eta_{\dg},\chi_{\dg})\notag\\
&\qquad-\se \left\langle D^2\eta_{\dg}\nu_E\right\rangle_E{\cdot}[\nabla\chi_{\dg}]_{E} \ds-\se \left\langle D^2\chi_{\dg}\nu_E\right\rangle_E{\cdot}[\nabla\eta_{\dg}]_{E} \ds\notag\\
&\qquad\quad+\sum_{E\in\cE}\left(\frac{\sigma_{\dg}}{h_E^3}\int_E[\eta_{\dg}]_E[\chi_{\dg}]_E\ds+\frac{\sigma_{\dg}}{h_E}\int_E\left[\frac{\partial\eta_{\dg}}{\partial\nu_E}\right]_E\left[\frac{\partial\chi_{\dg}}{\partial\nu_E}\right]_E\ds\right).
\end{align}
In general, the two stabilization terms in the bilinear form may rely on different penalty parameter. The DG norm is defined by $$\|\eta_{\dg}\|_{\dg}^2:=\|\eta_{\dg}\|_{\nc}^2+\sum_{E\in\cE}h_E^{-1}\left\|\left[\frac{\partial\eta_{\dg}}{\partial\nu_E}\right]_E\right\|_{L^2(E)}^2+\sum_{E\in\cE}h_E^{-3}\|\jump{\eta_{\dg}}\|_{L^2(E)}^2.$$
For sufficiently large penalty parameter $\sigma_{\dg}$, the coercivity result  \cite{CCGMNN_DG} $\|\bullet\|_{\dg}^2\lesssim a_{\dg}(\bullet,\bullet)$ on $P_2(\cT)$ holds.
A DGFEM corresponding to \eqref{vform_cts} seeks $\Psi_{\dg}\in \bDGS:=\DGS\times \DGS$ such that
\begin{equation}\label{vformd_dg}
N_\dg(\Psi_{\dg};\Phi_{\dg}):=A_{\dg}(\Psi_{\dg},\Phi_{\dg})+B_{h}(\Psi_{\dg},\Psi_{\dg},\Phi_{\dg})-L_{h}(\Phi_{\dg})=0 \fl \Phi_{\dg} \in  \bDGS,
\end{equation}
where for all $ \Xi=(\xi_{1},\xi_{2}),\Theta=(\theta_{1},\theta_{2})$ and $ \Phi=(\varphi_{1},\varphi_{2})\in\bDGS$, the bilinear form $A_{\dg}(\Theta,\Phi):=a_{\dg}(\theta_1,\varphi_1)+a_{\dg}(\theta_2,\varphi_2)$, trilinear and linear forms
$B_h(\Xi,\Theta,\Phi)$ and $L_h(\Phi)$ are respectively as defined in \eqref{defn_Bh} and \eqref{defn_Lh}. Moreover, the following boundedness result \cite[Lemma~3.12(a)]{CCGMNN_DG} holds
\begin{equation}
B_h(\Xi,\Theta,\Phi)\leq \trinl\Xi\trinr_{\dg}\trinl\Theta\trinr_{\dg}\trinl\Phi\trinr_{\dg}\quad \text{for } \Xi,\Theta,\Phi\in \bX+\bDGS.
\end{equation}
The existence, local uniqueness and error estimates for the discrete solution of \eqref{vformd_dg} are shown in \cite{CCGMNN_DG} for sufficiently small mesh parameter $h$.

\begin{lem}[Enrichment operator]\cite{Georgoulis2011,CCGMNN_DG}\label{enrichment_dG} There exists an enrichment operator $E_h: \St\to S_4(\cT)\subset X$ satisfies, for $m=0,1,2$ 
\begin{align}\label{enrich_apost}
\sum_{K\in\cT}\left|\varphi_{\dg}-E_h\varphi_{\dg}\right|_{H^m(K)}^2  &\lesssim\|h_{\cE}^{1/2-m}[\varphi_{\dg}]_{\cE}\|_{L^2(\Gamma)}^2+\|h_{\cE}^{3/2-m}[\nabla \varphi_{\dg}]_{\cE}\|_{L^2(\Gamma)}^2\lesssim h^{4-2m}\|\varphi_{\dg}\|_{\dg}^2,
\end{align}
where $S_4(\cT)$ is a $C^1$-conforming finite element space consisting of macro-elements.
\end{lem}

\section{Equivalence of approximations}

\subsection{Main results}

Throughout the paper, the oscillation of a function $f\in L^2(\Omega)$ with respect to a triangulation $\cT$ reads
\begin{equation*}
{\rm osc}(f,\cT):=\sqrt{\sum_{K\in\mathcal{T}}h_K^4\|f-\avintK f\dx\|_{L^2(K)}^2} \text{ with } \avintK f\dx:=\frac{1}{|K|}\int_K f\dx.
\end{equation*}
Also denote the local oscillation term by ${\rm osc}(f,K):=h_K^2\|f-\avintK f\dx\|_{L^2(K)} $.
\begin{thm}[Error equivalence]\label{error_equiv}
	For sufficiently small mesh parameter $h$, the discrete solutions $\Psi_{\M},\Psi_{\ip}$ and $\Psi_{\dg}$ of the Morley FEM, $C^0$ IP and DGFEM satisfy
	\begin{equation*}
	\trinl\Psi-\Psi_{\M}\trinr_{\nc}\approx \trinl\Psi-\Psi_{\ip}\trinr_{\ip}\approx\trinl\Psi-\Psi_{\dg}\trinr_{\dg}\approx \trinl(1-\Pi_0)D^2\Psi\trinr_{L^2(\Omega)}
	\end{equation*}
	up to some oscillations.
\end{thm}
The proof follows by the following results.

Let $\cN(E)$ be the set of two vertices of an edge $E$. Define the following seminorm found in \cite{CC_DG_NN_15_Comparison}  for all $v_h\in H^2(\cT)$ 
\begin{equation}
\|v_h\|_h^2:=\|v_h\|_{\nc}+\sum_{E\in\cE}\left(\avintE\left[\frac{\partial v_h}{\partial\nu_E}\right]_E\ds\right)^2+\sum_{E\in\cE}h_{E}^{-2}\sum_{z\in\cN(E)}[v_h(z)]_E^2,
\end{equation}
which is a norm on $X+P_2(\cT)$.

\begin{lem}[Discrete norm equivalence]\label{norm_equivalence}\cite[Theorem 4.1]{CC_DG_NN_15_Comparison} The norm $\|\bullet\|_h$ satisfies
	\begin{align*}
	\|\bullet\|_h&=\|\bullet\|_{\nc} \quad \text{ on } X+M(\cT),\\
	\|\bullet\|_h&\approx \|\bullet\|_{\dg} \quad \text{ on } X+P_2(\cT),\\
	\|\bullet\|_h&\approx \|\bullet\|_{\ip} \quad\; \text{ on } X+\ip(\cT).
	\end{align*}
\end{lem}

\begin{lem}[Equivalence of best approximations]\label{equi_best_apprx}\cite[Theorem 3.1]{CC_DG_NN_15_Comparison} For any $v\in X$, the following distances are equivalent	
	\begin{align}
	\min_{v_{\dg}\in P_2(\cT) }\|v-v_{\dg}\|_{h}=\min_{v_{\M}\in M(\cT)} \|v-v_{\M}\|_{h}\approx \min_{v_{\ip}\in \ip(\cT) }\|v-v_{\ip}\|_{h}.
	\end{align}
\end{lem}

The above equivalence result shows that the interpolations $I_{\M}\Psi$ for Morley element, and interpolations for $C^0IP$ and DGFEM satisfy the error equivalence. In \cite{CC_DG_NN_15_Comparison}, it has been shown that, in particular, the $P_2$ finite element approximations for linear biharmonic problem satisfies the error equivalence up to some data oscillation. However, it is not clear whether the computed finite element solutions $\Psi_{\M}, \Psi_{\dg}$ and $\Psi_{\ip}$ for the solution $\Psi$ of the semilinear \vket will follow the same equivalence result. We proceed to show that the error equivalence results are true for the computed FE solutions up to data oscillation, for sufficiently small mesh parameter $h$. In  the next theorem, an abstract error estimate results of \cite[]{CCGMNN_Semilinear} for nonconforming FEM is stated, and the result can be easily extended to $C^0$IP and DGFEMs as well, hence we omit the details.
\begin{thm}[Error estimates]\label{thm_abs_est}
Let $\Psi=(u,v)$ be the nonsingular solution of \eqref{vform_cts}. Let $\Psi_{\M}, \Psi_{\ip}$ and $\Psi_{\dg}$ be discrete solution of \eqref{vformd_nc}, \eqref{vformd_ip} and \eqref{vformd_dg} respectively. The errors of Morley, $C^0$IP and DG FEMs are quasi-optimal with respect to their norms in the sense that 
	\begin{align}
	\trinl\Psi-\Psi_{\M}\trinr_{\nc}&\lesssim \min_{\bv_{\M}\in \bMS}\trinl\Psi-\bv_{\M}\trinr_{\nc}+\|N_{\nc}(\Psi)\|_{\bMS^*},\label{BestApprxNC}\\
	\trinl\Psi-\Psi_{\ip}\trinr_{\ip}&\lesssim \min_{\bv_{\ip}\in \bIPS}\trinl\Psi-\bv_{\ip}\trinr_{\ip}+\|N_{\ip}(\Psi)\|_{\bIPS^*},\label{BestApprxIP}\\
	\trinl\Psi-\Psi_{\dg}\trinr_{\dg}&\lesssim \min_{\bv_{\M}\in \bDGS}\trinl\Psi-\bv_{\dg}\trinr_{\dg}+\|N_{\dg}(\Psi)\|_{\bDGS^*}\label{BestApprxDG}
	\end{align}
for sufficiently small mesh parameter $h$ of  $\cT$.	
\end{thm}

In the following lemmas, we compute the residuals $\|N_{\nc}(\Psi)\|_{\bMS^*},\|N_{\ip}(\Psi)\|_{\bIPS^*}$ and $\|N_{\dg}(\Psi)\|_{\bDGS^*}$ using the technique of {\it a posteriori} error estimates.
We start with some essential terms related to the residuals which are proved to be efficient. 
\begin{lem}\label{lem:vol_eff}
Let $\Psi=(u,v)$ be the nonsingular solution of \eqref{vform_cts}. For $\Theta_h=(\theta_{h,1},\theta_{h,2})\in\bDGS$, it holds
\begin{align}
 h_K^2\|f+[\theta_{h,1},\theta_{h,2}]\|_{L^2(K)}+ h_K^2\|[\theta_{h,1},\theta_{h,1}]\|_{L^2(K)}\lesssim\trinl \Psi-\Theta_h\trinr_{L^2(K)}+{\rm osc}(f,K).
\end{align}
\end{lem}
The above local volume efficiency can be established from \cite[Lemma~5.3]{CCGMNN_DG}, hence proof has been omitted.
\begin{lem}\label{lem:jump_eff}
	Let $\Psi=(u,v)$ be the nonsingular solution of \eqref{vform_cts}. For $\Theta_h=(\theta_{h,1},\theta_{h,2})\in\bDGS$, it holds
\begin{align}
&h_E^{1/2}\|\jump{D^2 \theta_{h,1}\,\nu_E}\cdot\nu_E\|_{L^2(E)}+h_E^{1/2}\|\jump{D^2 \theta_{h,2}\,\nu_E}\cdot\nu_E\|_{L^2(E)}\notag\\
&\quad\lesssim \trinl\Psi-\Theta_{h}\trinr_{H^2(\omega_E)}+{\rm osc}(f,\omega_E).
\end{align}
\end{lem}
\begin{proof}
First we  prove the efficiency of edge term $\|h_E^{1/2}\left[D^2 \theta_{h,1}\nu_E\right]_E\cdot\nu_E\|^2_{L^2(E)}$. For each internal edge $E\in \cE(\Omega)$,  define $\tilde{K}\subset \omega_E:=K_+\cup K_-$ to be the largest rhombus contained in the patch $\omega_E$ that has $E$ as one diagonal. Also, define $b_{\tilde{K}}:\tilde{K}\map \bR$ to be the bubble function on the rhombus $\tilde{K}$. Let $b_l:\tilde{K}\map \bR$ be an affine function having value zero along the edge $E$, such that $(\nabla b_l\cdot\nu_E)|_E=h_E^{-1}$. Using the above definitions,  consider the function $b_E$ with $b_E|_{\tilde{K}}:=b_l b_{\tilde{K}}^3$ and $b_E:=0$ on $\Omega\setminus\tilde{K}$, which has the following properties \cite{Georgoulis2011}:
\begin{align*}
& b_E\in C^2(\Omega)\cap\hto,\quad \nabla b_E\cdot\nu_E|_E=h_E^{-1}b_{\tilde{K}}^3|_E,\quad \nabla b_E\cdot\tau_E|_E=0.
\end{align*}
Extend $[D^2 \theta_{h,1}\nu_E]\cdot\nu_E$ constantly in the normal direction to $E$ and set $\rho_E:=h_E^{-1}[D^2 \theta_{h,1}\nu_E]\cdot\nu_E\, b_E$.

\noindent Incorporate the properties of $b_E$ with $\rho_E$ and integrate by parts to obtain
\begin{align*}
& \|h_E^{-1}[D^2 \theta_{h,1}\nu_E]_E\cdot\nu_E\|_{L^2(E)}^2  \lesssim  \|b_{\tilde{K}}^{{3/2}} h_E^{-1}[D^2 \theta_{h,1}\nu_E]_E\cdot\nu_E\|_{L^2(E)}^2\\
&=\int_E \{\nabla\rho_E\cdot\nu_E \}_E[D^2 \theta_{h,1}\nu_E]_E\cdot\nu_E\ds
=\int_{\omega_E} D^2 \theta_{h,1}:D^2\rho_E\dx -\int_{\omega_E}\Delta^2 \theta_{h,1}\rho_E\dx.
\end{align*}
Since $\Theta_h\in\bDGS$, the above equation and \eqref{wforma} leads to
\begin{align*}
&\|h_E^{-1}[D^2 \theta_{h,1}\nu_E]_E\cdot\nu_E\|_{L^2(E)}^2\leq\int_{\omega_E} D^2 (\theta_{h,1}-u):D^2\rho_E\dx +\int_{\omega_E} [u,v]\rho_E\dx+\int_{\omega_E} f\rho_E\dx\\
&=\int_{\omega_E} D^2 (\theta_{h,1}-u):D^2\rho_E\dx +\int_{\omega_E} \left([u,v]-[\theta_{h,1},\theta_{h,2}]\right)\rho_E\dx+\int_{\omega_E}\left(f+[\theta_{h,1},\theta_{h,2}]\right) \rho_E\dx.
\end{align*}
The first term is estimated by Cauchy and inverse inequalities
\begin{equation*}
\int_{\omega_E} D^2 (\theta_{h,1}-u):D^2\rho_E\dx\leq\|u-\theta_{h,1}\|_{H^2(\omega_E)}\|\rho_E\|_{H^2(\omega_E)}\lesssim\|u-\theta_{h,1}\|_{\htk}\|h_E^{-2}\rho_E\|_{L^2(\omega_E)}.
\end{equation*}
The second term is estimated generalized \Holder inequality  
\begin{align*}
&\int_{\omega_E} \left([u,v]-[\theta_{h,1},\theta_{h,2}]\right)\rho_E\dx=-2b(u,v,\rho_E)+2b(\theta_{h,1},\theta_{h,2},\rho_E)\\
&=-2b(u-\theta_{h,1},v,\rho_E)-2b(\theta_{h,1},v-\theta_{h,2},\rho_E)\notag \\
&\lesssim \|\Psi-\Theta_{h}\|_{H^2(\omega_E)}(\|\Psi\|_{H^2(\omega_E)}+\|\Theta_{h}\|_{H^2(\omega_E)})\|\rho_E\|_{L^{\infty}(\omega_E)}\lesssim \|\Psi-\Theta_{h}\|_{H^2(\omega_E)}\|\rho_E\|_{H^2(\omega_E)}\notag\\
&\lesssim\|\Psi-\Theta_{h}\|_{H^2(\omega_E)}\|h_E^{-2}\rho_E\|_{L^2(\omega_E)}.
\end{align*}
The last term is estimated as
\begin{align*}
&\int_{\omega_E}\left(f+[\theta_{h,1},\theta_{h,2}]\right)\rho_E\dx\leq h_E^2\|f+[\theta_{h,1},\theta_{h,2}]\|_{L^2(\omega_E)}\|h_E^{-2}\rho_E\|_{L^2(\omega_E)}.
\end{align*}
The efficiency of the volume term $\eta_K$ and above displayed equations lead to
\begin{align}
\|h_E^{-1}[D^2 \theta_{h,1}\nu_E]_E\cdot\nu_E\|_{L^2(E)}^2&\lesssim\left( \|\Psi-\Theta_{h}\|_{H^2(\omega_E)}+{\rm osc}(f,\omega_E)\right)\|h_E^{-2}\rho_E\|_{L^2(\omega_E)}.\label{estder2}
\end{align}

\noindent Let $l(s)$ denote the length of the intersection of the line normal to $E$, crossing $E$ at the point $s\in E$ and $\tilde{K}$. Then
\begin{align}
\|\rho_E\|_{L^2(\omega_E)}&\lesssim h_E^{-1}\|[D^2 \theta_{h,1}\nu_E]_E\cdot\nu_E\|_{L^2(\omega_E)}=h_E^{-1}\left(\int_E\left|[D^2 \theta_{h,1}\nu_E]_E\cdot\nu_E\right|^2 l(s)\ds\right)^{1/2}\notag\\
&\lesssim \|h_E^{-1/2}[D^2 \theta_{h,1}\nu_E]_E\cdot\nu_E\|_{L^2(E)}.\label{rhoest}
\end{align}
Combine \eqref{estder2} and \eqref{rhoest}  to obtain
\begin{equation*}
h_E^{1/2}\|[D^2 \theta_{h,1}\nu_E]_E\cdot\nu_E\|_{L^2(E)}\lesssim \|\Psi-\Theta_{h}\|_{H^2(\omega_E)}\|\Psi\|_{H^2(\omega_E)}+{\rm osc}(f,\omega_E).
\end{equation*}
A similar argument leads to an estimate for the edge term
\begin{equation*}
h_E^{1/2}\|[D^2 \theta_{h,2}\nu_E]_E\cdot\nu_E\|_{L^2(E)}\lesssim \|\Psi-\Theta_{h}\|_{H^2(\omega_E)}\|\Psi\|_{H^2(\omega_E)}+{\rm osc}(f,\omega_E).
\end{equation*}
The  above two equations complete the proof.
\end{proof}

The next theorem establishes that the residual term for DG method are equivalent to best approximation up to a higher-order oscillation.
\begin{thm}\label{thm_const_osc} 
	Let $\Psi=(u,v)$ be the nonsingular solution of \eqref{vform_cts}. The consistency term $\|N_{\dg}(\Psi)\|_{\bDGS^*}$ has the estimate
	\begin{equation}\label{consist_osc}
	\|N_{\dg}(\Psi)\|_{\bDGS^*}\lesssim \min_{\bv_{\dg}\in \bDGS}\trinl\Psi-\bv_{\dg}\trinr_{\dg}+{\rm osc}(f,\cT).
	\end{equation}		
\end{thm}
\begin{proof}
	Let $\Psi_{\dg}^*=(u_{\dg}^*,v_{\dg}^*)\in\bDGS$ be the best approximation of $\Psi=(u,v)$ with respect to DG norm $\trinl\bullet\trinr_{\dg}$ in the discrete space $\bDGS$.  There exists $\Phi_{\dg}=(\phi_{\dg,1},\phi_{\dg,2})\in\bDGS$ with $\trinl\Phi_{\dg}\trinr_{\dg}=1$ such that $\|N_{\dg}(\Psi)\|_{\bDGS^*}=N_{\dg}(\Psi,\Phi_{\dg})$. Denote $\boldsymbol{\chi}=(\chi_1,\chi_2):=\Phi_{\dg}-E_{\dg}\Phi_{\dg}$. Since $\Psi\in \bX\cap {\bf H}^{2+\alpha}(\Omega)$ is the solution of \eqref{vform_cts}, $N_{\dg}(\Psi,E_{\dg}\Phi_{\dg})=N(\Psi,E_{\dg}\Phi_{\dg})=0$. This implies $N_{\dg}(\Psi,\Phi_{\dg})=N_{\dg}(\Psi,\Phi_{\dg}-E_{\dg}\Phi_{\dg})=N_{\dg}(\Psi,\boldsymbol{\chi})$. The definition of $N_{\dg}$ reads
	\begin{align}\label{defn_Nh}
	&N_{\dg}(\Psi,\boldsymbol{\chi})=A_{\dg}(\Psi,\boldsymbol{\chi})+B_{h}(\Psi,\Psi,\boldsymbol{\chi})-F(\boldsymbol{\chi})\notag\\
	&=A_{\dg}(\Psi-\Psi_{\dg}^*,\boldsymbol{\chi})+\left(B_{h}(\Psi,\Psi,\boldsymbol{\chi})-B_{h}(\Psi_{\dg}^*,\Psi_{\dg}^*,\boldsymbol{\chi})\right)\notag\\
	&\quad+A_{\dg}(\Psi_{\dg}^*,\boldsymbol{\chi})+B_{h}(\Psi_{\dg}^*,\Psi_{\dg}^*,\boldsymbol{\chi})-F(\boldsymbol{\chi}).
	\end{align}
	First we obtain a bound for $A_{\dg}(\Psi_{\dg}^*,\boldsymbol{\chi})$ and then we proceed for  the estimate \eqref{consist_osc}. Start with one of the component $a_{\dg}(u_{\dg}^*,\chi_1)$ of $A_{\dg}(\Psi_{\dg}^*,\boldsymbol{\chi})$ as:
	\begin{align}
	&a_{\dg}(u_{\dg}^*,\chi_1)=\sik D^2u_{\dg}^*:D^2\chi_1\dx-\se\avg{D^2u_{\dg}^*\nu_E}\cdot\jump{\nabla \chi_1}\notag\\
	&-\se\avg{D^2\chi_1\nu_E}\cdot\jump{\nabla u_{\dg}^*}\ds+\sum_{E\in\cE}\frac{\sigma_{\dg}}{h_E}\int_E\jump{\nabla u_{\dg}^*}\cdot\jump{\nabla\chi_1}\ds+\sum_{E\in\cE}\frac{\sigma_{\dg}}{h_E^3}\int_E\jump{ u_{\dg}^*}\jump{\chi_1}\ds.\label{adg_expand}
	\end{align}
	Since $u_{\dg}^*\in P_2(\cT)$ is piecewise quadratic polynomial, an integration by parts for the first term of the above equation yields
	\begin{align}\label{int_by_parts}
	\sik D^2u_{\dg}^*:D^2\chi_1\dx=\se\avg{D^2u_{\dg}^*\nu_E}\cdot\jump{\nabla \chi_1}\ds+\sie\jump{D^2u_{\dg}^*\nu_E}\cdot\avg{\nabla \chi_1}\ds.
	\end{align}
	Using the above equations \eqref{adg_expand}-\eqref{int_by_parts}, we obtain
	\begin{align}
	a_{\dg}(u_{\dg}^*,\chi_1)&=\sie\jump{D^2u_{\dg}^*\nu_E}\cdot\avg{\nabla \chi_1}\ds-\se\avg{D^2\chi_1\nu_E}\cdot\jump{\nabla u_{\dg}^*}\ds\notag\\
	&\quad+\sum_{E\in\cE}\frac{\sigma_1}{h_E}\int_E\jump{\nabla u_{\dg}^*}\cdot\jump{\nabla\chi_1}\ds+\sum_{E\in\cE}\frac{\sigma_1}{h_E^3}\int_E\jump{ u_{\dg}^*}\jump{\chi_1}\ds.\label{adg_simple}
	\end{align}
	A use of the gradient representation $\nabla \chi_1=\frac{\partial\chi_1}{\partial\nu}\nu+\frac{\partial\chi_1}{\partial\tau}\tau$, in the first term of the above yields
	\begin{align}
	\sie\jump{D^2u_{\dg}^*\nu_E}\cdot\avg{\nabla \chi_1}\ds&=\sie\jump{D^2u_{\dg}^*\nu_E}\cdot\nu_E\avg{\frac{\partial\chi_1}{\partial\nu}}\ds\notag\\
	&\quad+\sie\jump{D^2u_{\dg}^*\nu_E}\cdot\tau_E\avg{\frac{\partial\chi_1}{\partial\tau}}\ds.\label{jump_split}
	\end{align}
   The second term of the above equation is estimated by the Cauchy--Schwarz inequality as
\begin{align}
&\sie\jump{D^2u_{\dg}^*\nu_E\cdot\tau}\avg{\frac{\partial\chi_1}{\partial\tau}}\ds \notag \\
&\qquad \leq \bigg{(}\sum_{E\in\cE(\Omega)}\|h_E^{1/2}\jump{D^2 u_{\dg}^*\,\nu_E}\cdot\tau_E\|_{L^2(E)}^2\bigg{)}^{1/2}\|h_{\cE}^{-1/2}\left\langle {\partial\chi_1}/{\partial\tau}\right\rangle_{\cE}\|_{L^2(\Gamma)}.\label{eqn_D2unu}
\end{align}
Set $\psi_E(s):=\left[\frac{\partial u_{\dg}^*}{\partial \nu}\right]_E$ on $E\in\cE(\Omega)$. {An inverse inequality implies}
\begin{align}
&\|h_E^{1/2}[D^2 u_{\dg}^*\,\nu_E]_E\cdot\tau_E\|_{L^2(E)}=\|h_E^{1/2}\frac{\partial\psi_E}{\partial s}\|_{L^2(E)}\notag\\
&\qquad\lesssim \|h_E^{-1/2}\psi_E\|_{L^2(E)}=h_E^{-1/2}\|[\jump{\nabla u_{\dg}^*\cdot\nu_E}\|_{L^2(E)}\leq h_E^{-1/2}\|\jump{\nabla (u-u_{\dg}^*)\cdot\nu_E}\|_{L^2(E)}.\label{eqn_mix_der}
\end{align}
\noindent The trace inequality  and enrichment lemma lead to
\begin{align}
&\|h_{\cE}^{-1/2}\left\langle {\partial\chi_1}/{\partial\tau}\right\rangle_{\cE}  \|_{L^2(\Gamma)}^2\lesssim\sum_{K\in\cT}h_K^{-1}\|\nabla\chi_1\|_{L^2(\partial K)}^2\notag\\
&\qquad \lesssim\sum_{K\in\cT}h_K^{-1}\left(h_K^{-1}\|\chi_1\|_{H^1(K)}^2+h_K\|\chi_1\|_{H^2(K)}^2\right)\lesssim\trinl\phi_{\dg,1}\trinr_{\dg}^2\leq 1. \label{est_int_avg}
\end{align}
The above two equations \eqref{eqn_mix_der} and \eqref{est_int_avg} yield the estimate 
\begin{align}
&\sie\jump{D^2u_{\dg}^*\nu_E\cdot\tau}\avg{\frac{\partial\chi_1}{\partial\tau}}\ds\lesssim\|u-u_{\dg}^*\|_{\dg}.\label{est_eqn_D2unu}
\end{align}
The Cauchy-Schwarz inequality leads to an estimate for the first term of \eqref{jump_split} as
\begin{align}
&\sie\jump{D^2u_{\dg}^*\nu_E}\cdot\nu_E\avg{\frac{\partial\chi_1}{\partial\nu}}\ds\notag\\ 
&\qquad \leq \bigg{(}\sum_{E\in\cE(\Omega)}\|h_E^{1/2}\jump{D^2 u_{\dg}^*\,\nu_E}\cdot\nu_E\|_{L^2(E)}^2\bigg{)}^{1/2}\|h_{\cE}^{-1/2}\left\langle {\partial\chi_1}/{\partial\tau}\right\rangle_{\cE}\|_{L^2(\Gamma)}.\label{eqn_D2unn}
\end{align}
The Lemma~\ref{lem:jump_eff} of efficiency yields
\begin{align}
&\bigg{(}\sum_{E\in\cE(\Omega)}\|h_E^{1/2}\jump{D^2 u_{\dg}^*\,\nu_E}\cdot\nu_E\|_{L^2(E)}^2\bigg{)}^{1/2}\lesssim \|u-u_{\dg}^*\|_{\dg}+{\rm osc}(f,\cT).
\end{align}
The Cauchy-Schwarz, the trace inequality and enrichment  Lemma~\ref{enrichment_dG} yield an estimate for second term of \eqref{adg_simple} :
\begin{align}
\se\avg{D^2\chi_1\nu_E}\cdot\jump{\nabla u_{\dg}^*}\ds\lesssim\| \phi_{\dg,1}\|_{\dg}\|h_\cE^{-1/2}[\nabla u_{\dg}^*]_E\|_{\Gamma}\leq \|h_\cE^{-1/2}[\nabla u_{\dg}^*]_E\|_{\Gamma}. 
\end{align}
Combining the estimates for first two terms of \eqref{adg_simple} and adding-subtracting $u$ in the jump terms of $\jump{\nabla u_{\dg}^*}$ and $\jump{u_{\dg}^*}$, we obtain
\begin{equation}
a_{\dg}(u_{\dg}^*,\chi_1)\lesssim \|u-u_{\dg}^*\|_{\dg}+{\rm osc}(f,\cT).
\end{equation}
Similar result for $a_{\dg}(v_{\dg}^*,\chi_2)$ yields an estimate
\begin{equation}
A_{\dg}(\Psi_{\dg}^*,\boldsymbol{\chi})\lesssim \trinl\Psi-\Psi_{\dg}^*\trinr_{\dg}+ {\rm osc}(f,\cT).
\end{equation}
	The Cauchy-Schwarz and enrichment lemma lead to
	\begin{equation}
	B_{\dg}(\Psi_{\dg}^*,\Psi_{\dg}^*,\boldsymbol{\chi})-F(\boldsymbol{\chi})\lesssim\sum_{K\in\cT}h_K^2\|f+[u_{\dg}^*,v_{\dg}^*]\|_{L^2(K)}+\sum_{K\in\cT}h_K^2\|[u_{\dg}^*,u_{\dg}^*]\|_{L^2(K)}.
	\end{equation}
	The above two displayed equations with Lemma~\ref{lem:vol_eff} imply
	\begin{align}
	&A_{\dg}(\Psi_{\dg}^*,\boldsymbol{\chi})+B_{\dg}(\Psi_{\dg}^*,\Psi_{\dg}^*,\boldsymbol{\chi})-F(\boldsymbol{\chi})\lesssim \min_{\bv_{\dg}\in \bDGS}\trinl\Psi-\bv_{\dg}\trinr_{\dg}+{\rm osc}(f,\cT).
	\end{align}
    The boundedness of  $A_{\dg}(\bullet,\bullet)$ and $B_h(\bullet,\bullet,\bullet)$, and Lemma~\ref{enrichment_dG} yield
    \begin{equation}
    A_{\dg}(\Psi-\Psi_{\dg}^*,\boldsymbol{\chi})+\left(B_{h}(\Psi,\Psi,\boldsymbol{\chi})-B_{h}(\Psi_{\dg}^*,\Psi_{\dg}^*,\boldsymbol{\chi})\right)\lesssim \trinl\Psi-\Psi_{\dg}^*\trinr_{\dg}\trinl\boldsymbol{\chi}\trinr_{\dg}\lesssim\trinl\Psi-\Psi_{\dg}^*\trinr_{\dg}.
    \end{equation}
	Since $\Phi_{\dg}\in\bDGS$ is arbitrary, this completes the proof.
\end{proof}

The next result for $C^0$IP method follows exactly similar way. 
\begin{cor}\label{cor_const_IP}
	Let $\Psi=(u,v)$ be the nonsingular solution of \eqref{vform_cts}. The consistency term $\|N_{\ip}(\Psi)\|_{\bIPS^*}$ has the estimate
\begin{equation}\label{consist_osc_IP}
\|N_{\ip}(\Psi)\|_{\bIPS^*}\lesssim \min_{\bv_{\ip}\in \bIPS}\trinl\Psi-\bv_{\ip}\trinr_{\ip}+{\rm osc}(f,\cT).
\end{equation}		
\end{cor}

Following \cite[Theorem 5.3]{CCGMNN_Semilinear} and last few steps of the above Theorem~\ref{thm_const_osc}, the next results follows immediately.
\begin{cor}\label{cor_const_M}
	Let $\Psi=(u,v)$ be the nonsingular solution of \eqref{vform_cts}. The consistency term $\|N_{\nc}(\Psi)\|_{\bMS^*}$ has the estimate
	\begin{equation}\label{consist_osc_M}
	\|N_{\nc}(\Psi)\|_{\bMS^*}\lesssim \min_{\bv_{\M}\in \bMS}\trinl\Psi-\bv_{\M}\trinr_{\nc}+{\rm osc}(f,\cT).
	\end{equation}		
\end{cor}

\begin{rem}
	We observe that \cite[Theorem 5.3]{CCGMNN_Semilinear} proves an a priori result
	\begin{equation*}
	\trinl\Psi-\Psi_{\M}\trinr_{\nc}\lesssim \min_{\bv_{\M}\in \bMS}\trinl\Psi-\bv_{\M}\trinr_{\nc}+{\rm osc}(f+[u,v],\cT)+{\rm osc}([u,u],\cT)\lesssim h^{\alpha},
	\end{equation*}
	where $\alpha\in (1/2,1]$ is the index of elliptic regularity. This involves oscillation term in unknowns $u,v$. Whereas the combination of Theorem~\ref{thm_abs_est} and \ref{cor_const_M} avoid the involvement of unknowns $u,v$ in the oscillation.	 
\end{rem}

{\it Proof of Theorem~\ref{error_equiv}}:  Previous Theorem~\ref{thm_abs_est}-\ref{thm_const_osc} and Corollary~\ref{cor_const_IP}-\ref{consist_osc_M} imply the equivalent best approximation result in the unified norm $\trinl\bullet\trinr_{h}$ as
\begin{align*}
\trinl\Psi-\Psi_{\M}\trinr_{\nc}&\approx \min_{\bv_{\M}\in \bMS}\trinl\Psi-\bv_{\M}\trinr_{h},\\
\trinl\Psi-\Psi_{\ip}\trinr_{\ip}&\approx \min_{\bv_{\ip}\in \bIPS}\trinl\Psi-\bv_{\ip}\trinr_{h},\\
\trinl\Psi-\Psi_{\dg}\trinr_{\dg}&\approx \min_{\bv_{\M}\in \bDGS}\trinl\Psi-\bv_{\dg}\trinr_{h},
\end{align*}
up to the oscillation ${\rm osc}(f,\cT)$. Then Lemma~\ref{equi_best_apprx} establishes 
	\begin{equation*}
\trinl\Psi-\Psi_{\M}\trinr_{\nc}\approx \trinl\Psi-\Psi_{\ip}\trinr_{\ip}\approx\trinl\Psi-\Psi_{\dg}\trinr_{\dg}\approx \min_{\bv_{\M}\in \bMS}\trinl\Psi-\bv_{\M}\trinr_{\nc}.
\end{equation*}
The interpolation Lemma~\ref{Morley_Interpolation} (a) shows $ \min_{\bv_{\M}\in \bMS}\trinl\Psi-\bv_{\M}\trinr_{\nc}=\trinl (I-\Pi_0)D^2\Psi\trinr_{L^2(\Omega)}$, and this completes the proof of main result.\qed

\section{Numerical Experiments}
Three examples are presented below for the numerical approximations of Morley FEM, $C^0$IP and DGFEM to illustrate that they are equivalent. In the following experiments stabilization parameter for $C^0$IP and DGFEM are set as $\sigma_{\ip}=\sigma_{\dg}=20$.
\subsection{Analytic solution}\label{sec:analytic}
We consider the exact solution $u(x,y)=\sin^2(\pi x)\sin^2(\pi y)$ and $v(x,y)=x^2y^2(1-x)^2(1-y)^2$ for \eqref{vke} on the unit square $\Omega$ with regularity index $\alpha=1$ and the corresponding data $f:=\Delta^2 u -[u,v]$ and $g:=\Delta^2 v+\half [u,u]$. The  numerical experiments are performed on a sequence of uniform meshes starting with an initial mesh $\cT_0$ (see Figure~) on a unit square domain. In the uniform refinement  process, each triangle is divided into four similar triangles, see Figure~\ref{fig:square_meshes}. The convergence histories for nonconforming, DG and $C^0$IP methods are shown in Figure~\ref{fig:conv_analytic}. It illustrates that the convergence rate for all the $P_2$ elements are close to $0.5$ with respect to number of degrees of freedom (\texttt{ndof}), i.e., $\trinl\Psi-\Psi_{\M}\trinr_{h}\approx \trinl\Psi-\Psi_{\dg}\trinr_{h}\approx \trinl\Psi-\Psi_{\ip}\trinr_{h}\approx \texttt{ndof}^{-1/2}$, when meshes are sufficiently refined.

\begin{figure}
	\begin{center}
		\subfloat[]{\includegraphics[width=0.45\textwidth]{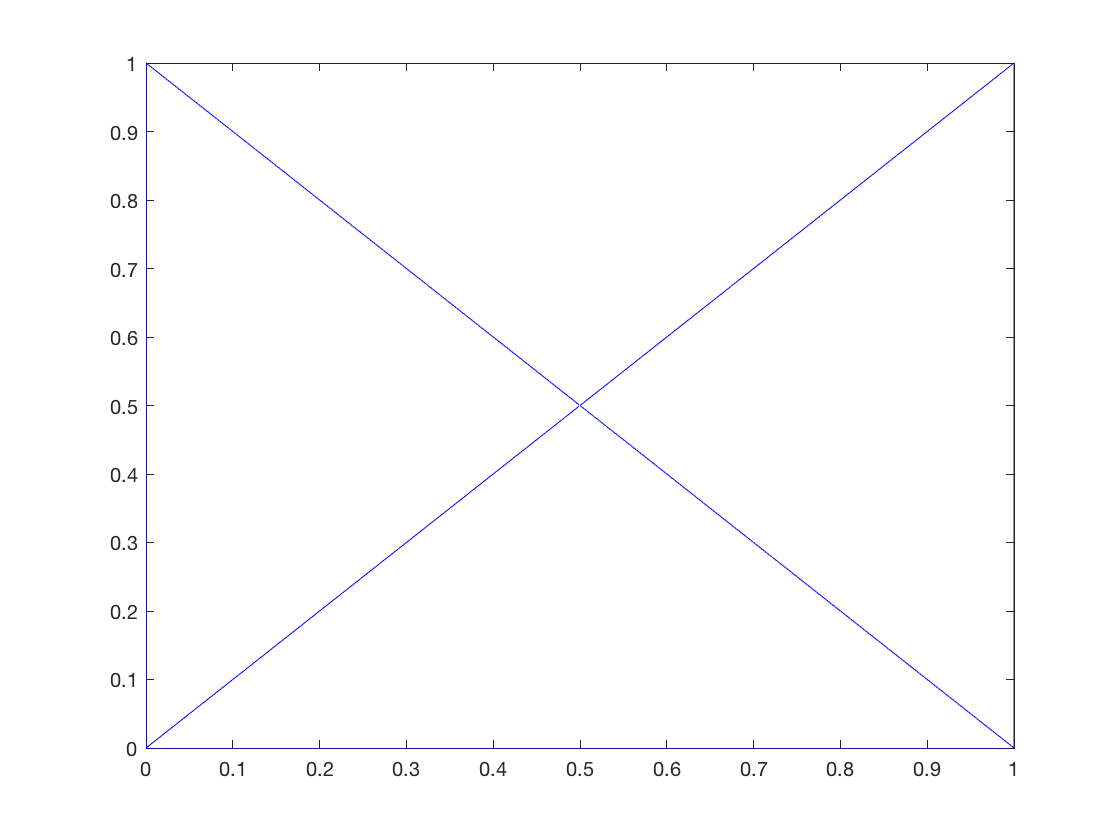}}
		\subfloat[]{\includegraphics[width=0.45\textwidth]{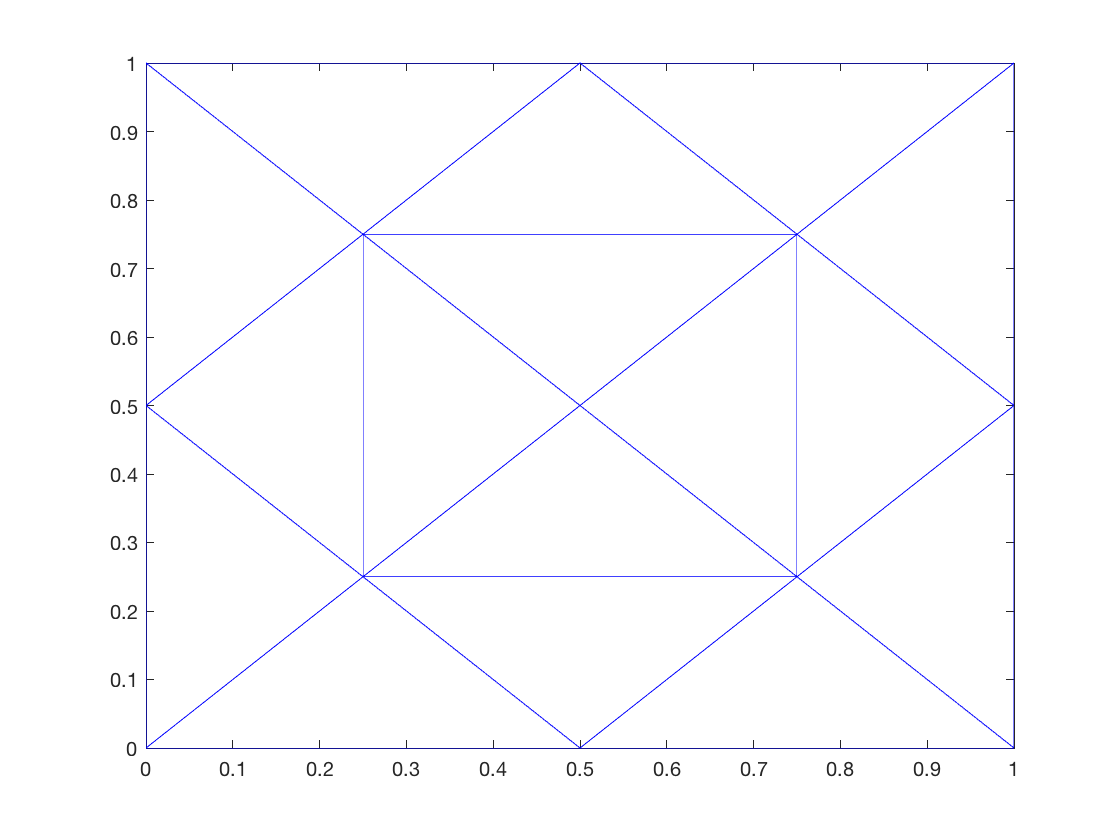}}
	\end{center}
	\caption{ (a) Initial mesh $\cT_0$ and (b) first uniform refinement $\cT_1$.}
	\label{fig:square_meshes}
\end{figure}
\begin{figure}
	\begin{center}
		\includegraphics[width=0.7\textwidth]{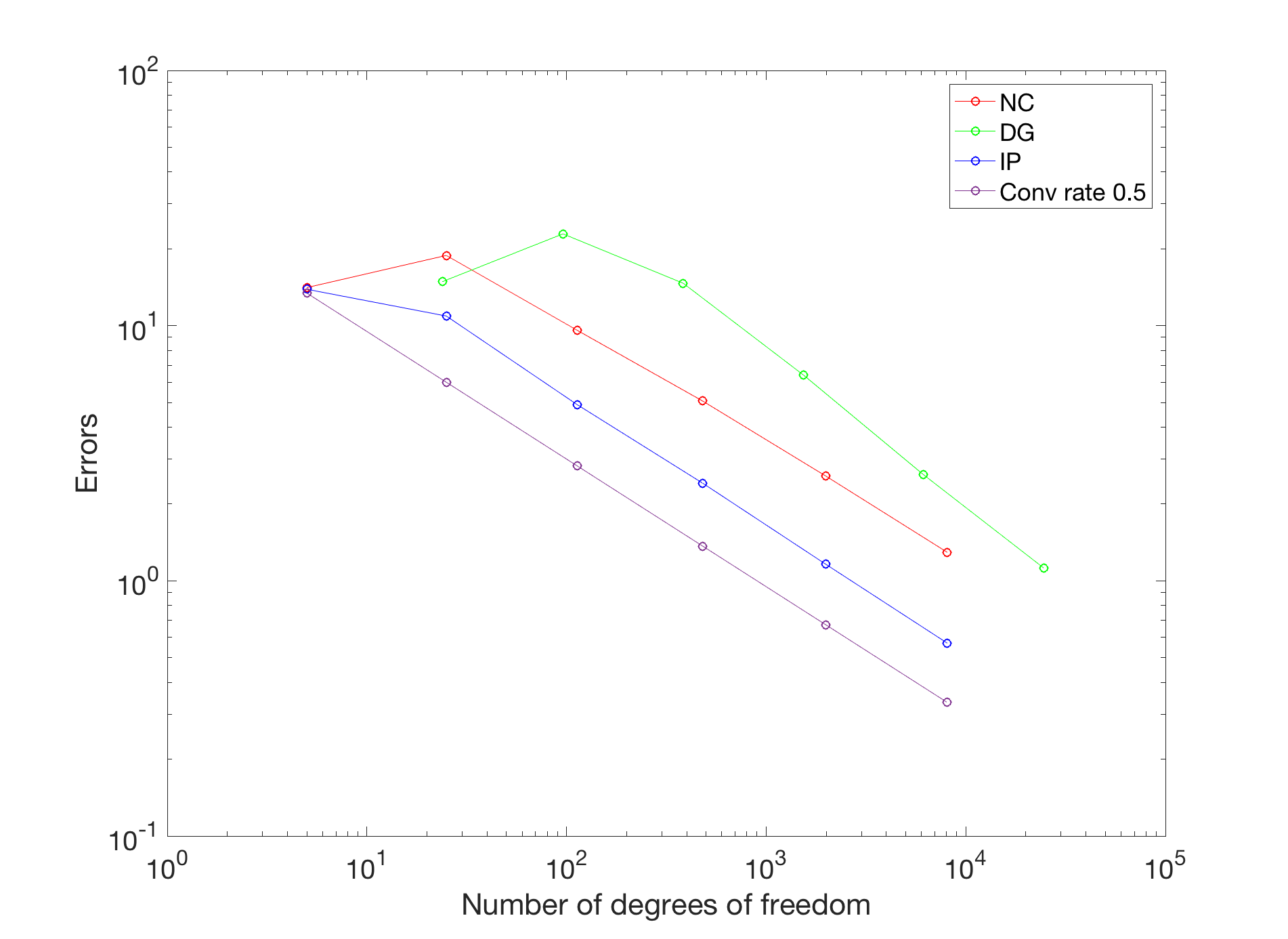}
	\end{center}
\caption{Convergence histories for Nonconforming, DG and $C^0$IP method on uniform meshes for Example~\ref{sec:analytic}. }
\label{fig:conv_analytic}
\end{figure}

\subsection{Singular solution on L-shaped domain with uniform refinement}\label{sec:sing_unif}
Consider the L-shaped domain  $\Omega=(-1,1)^2 \setminus\big{(}[0,1)\times(-1,0]\big{)}$. Set the
singular functions \cite{Grisvard}
$\displaystyle
u(r,\theta)=v(r,\theta):=(1-r^2 \cos^2\theta)^2 (1-r^2 \sin^2\theta)^2 r^{1+\alpha}g_{\alpha,\omega}(\theta)$ 	with
$	g_{\alpha,\omega}(\theta):=$
\begin{align*}
&\left(\frac{1}{\alpha-1}\sin\big{(}(\alpha-1)\omega\big{)}-\frac{1}{\alpha+1}\sin\big{(}(\alpha+1)\omega\big{)}\right)\times\Big{(}\cos\big{(}(\alpha-1)\theta\big{)}-\cos\big{(}(\alpha+1)\theta\big{)}\Big{)}\\
&-\left(\frac{1}{\alpha-1}\sin\big{(}(\alpha-1)\theta\big{)}-\frac{1}{\alpha+1}\sin\big{(}(\alpha+1)\theta\big{)}\right)\times\Big{(}\cos\big{(}(\alpha-1)\omega\big{)}-\cos\big{(}(\alpha+1)\omega\big{)}\Big{)},
\end{align*}
where the angle $\omega:=\frac{3\pi}{2}$ and the parameter $\alpha= 0.5444837367$ is a non-characteristic root of $\sin^2(\alpha\omega) = \alpha^2\sin^2(\omega)$. The loads $f$ and $g$ are chosen according to \eqref{vke}.  The  numerical experiments are performed on a sequence of uniform meshes starting with an initial mesh $\cT_0$ (see Figure~) on a L-shaped domain. The convergence histories for nonconforming, DG and $C^0$IP methods are shown in Figure~\ref{fig:conv_sing_unif}. It illustrates that the errors  for all the $P_2$ elements decay in same sub-optimal rate, and satisfy the equivalence of convergence $\trinl\Psi-\Psi_{\M}\trinr_{h}\approx \trinl\Psi-\Psi_{\dg}\trinr_{h}\approx \trinl\Psi-\Psi_{\ip}\trinr_{h}$, when meshes are sufficiently refined.

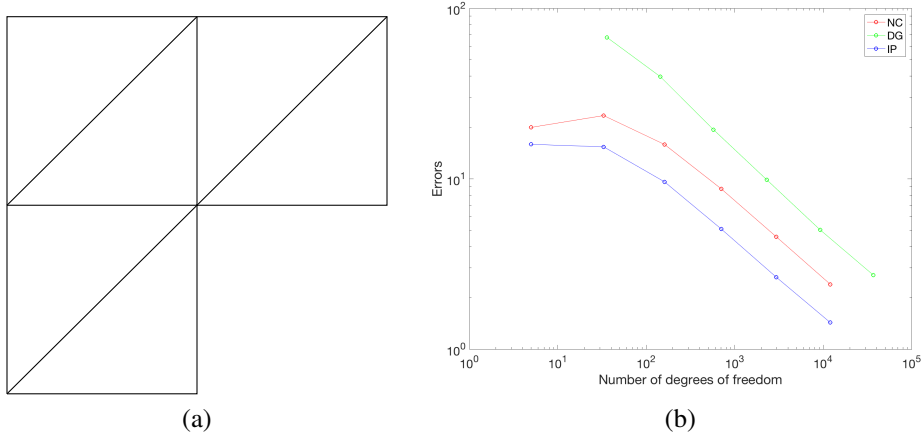
\begin{figure}
	\begin{center}
		\subfloat[]{
	      \begin{tikzpicture}
         	\draw[scale=2.5]  (0,0)--(1,0)--(1,1)--(2,1)--(2,2)--(0,2)--(0,0);
         	\draw[scale=2.5]  (0,0)--(1,1)--(2,2);
         	\draw[scale=2.5]  (0,1)--(1,2);
         	\draw[scale=2.5]  (0,1)--(1,1)--(1,2);
        	\end{tikzpicture}	
        	}
		\subfloat[]{\includegraphics[width=0.5\textwidth]{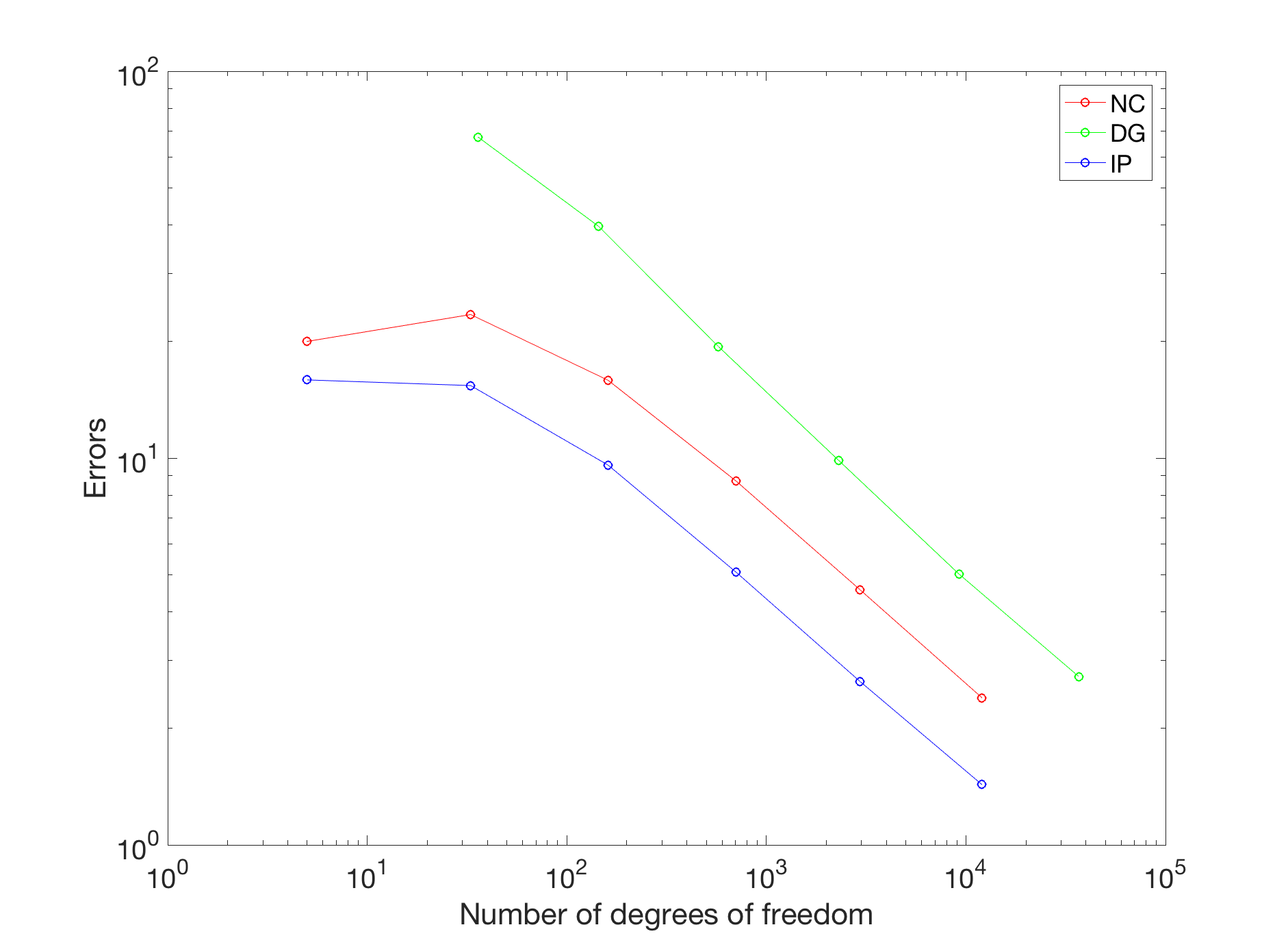}}	
	\end{center}
	\caption{(a) Initial L-shaped mesh and (b) Convergence histories for Nonconforming, DG and $C^0$IP method on uniform meshes for Example~\ref{sec:sing_unif}. }
	\label{fig:conv_sing_unif}
\end{figure}

\subsection{Singular solution on L-shaped domain with adaptive mesh refinement}\label{sec:sing_adap}
In this  test, we consider load functions $f$ and $g$ from the above Example~\ref{sec:sing_unif}, and perform adaptive refinement procedure from the initial mesh $\cT_0$. We follow the adaptive Algorithm~\ref{algo_adaptive} for each estimator of Morley FEM, DGFEM and $C^0$IP FEM  and this generates sequence of adaptive meshes with bulk parameter $\theta=0.5$. The convergence histories for nonconforming, DG and $C^0$IP methods with estimator from each of the method are shown in Figure~\ref{fig:conv_sing_adap}. This shows that adaptive mesh refinements lead to optimal convergence rate $0.5$ for $P_2$ FEMs. Moreover, for all the different estimators, Morley FEM, $C^0$IP and DGFEM show the same convergence rate, and this proves the equivalence of errors.


\begin{figure}
	\begin{center}
		 \subfloat[]{\includegraphics[width=0.35\textwidth]{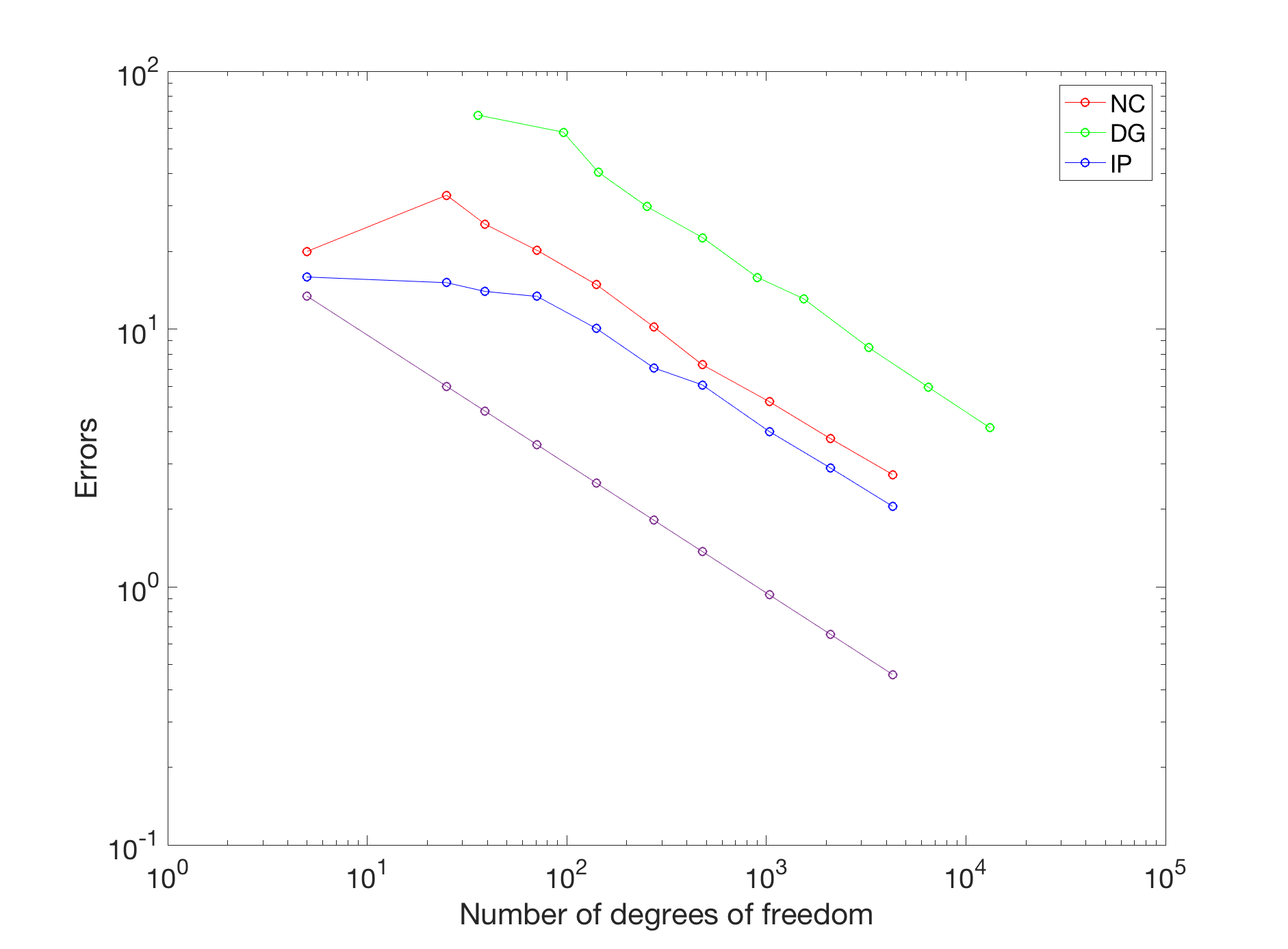}}
		 \subfloat[]{\includegraphics[width=0.35\textwidth]{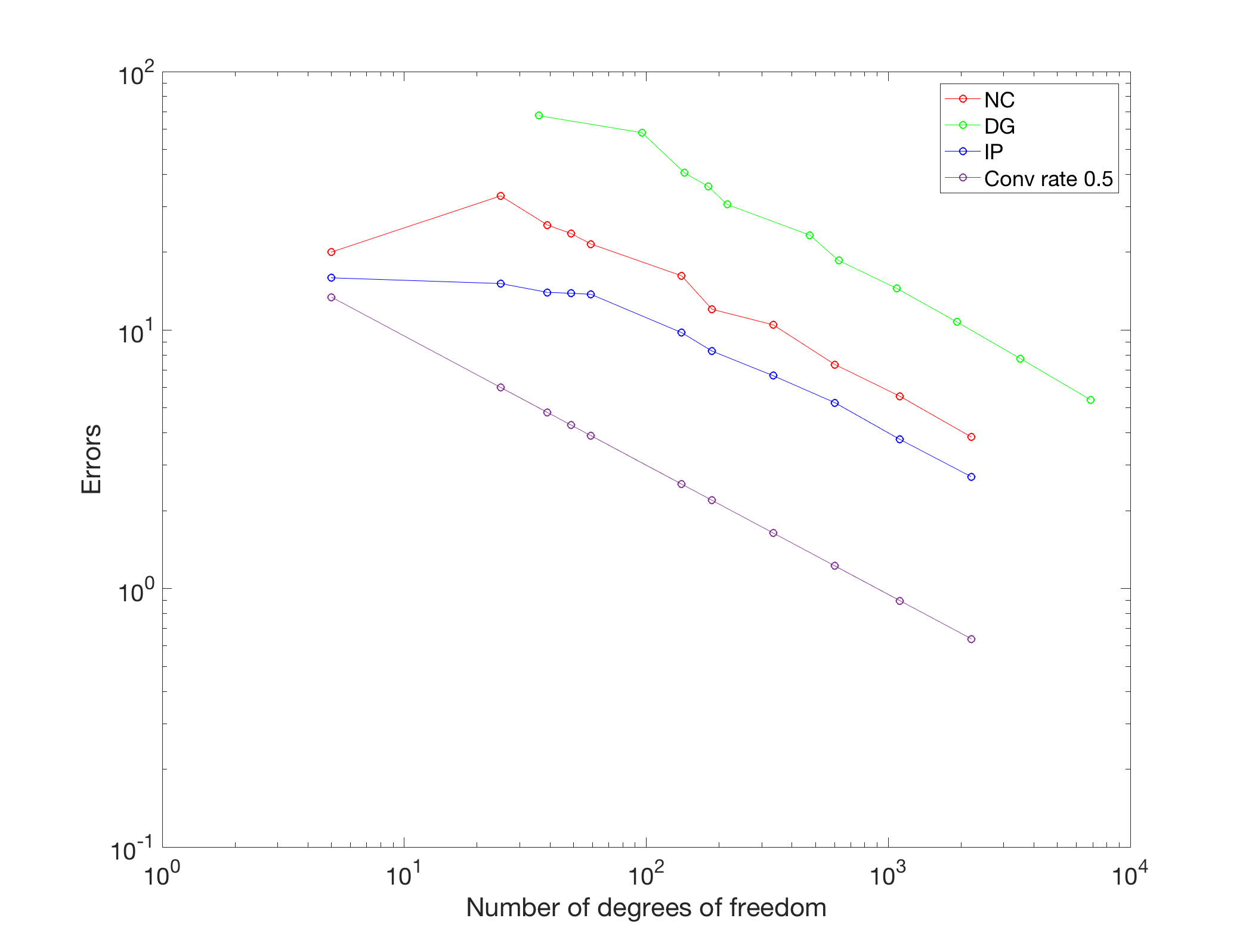}}
		 \subfloat[]{\includegraphics[width=0.35\textwidth]{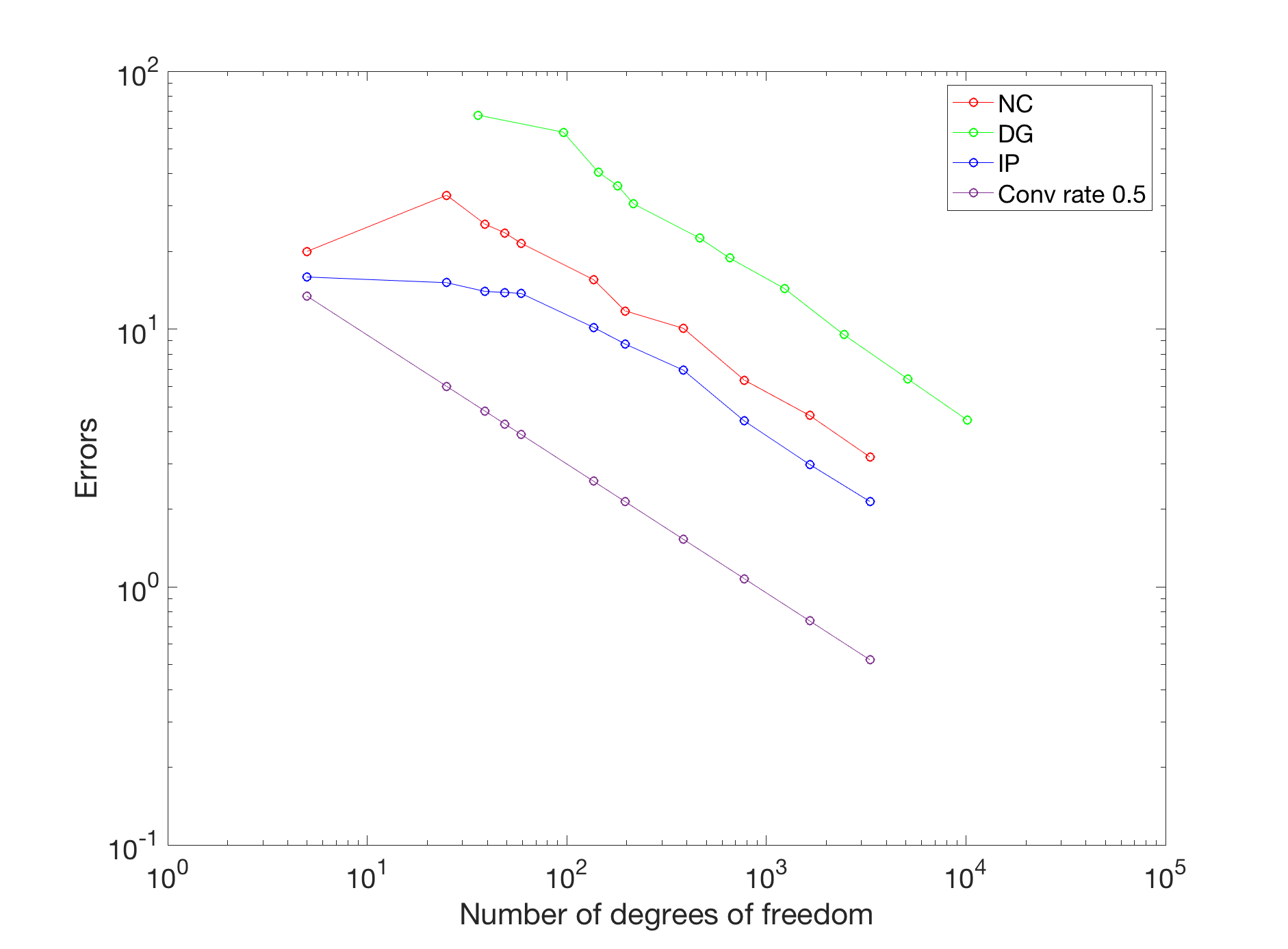}}
	\end{center}
	\caption{Convergence histories for Nonconforming, DG and $C^0$IP method on adaptive meshes for Example~\ref{sec:sing_adap} with (a) Morley FEM estimator, (b) DGFEM estimator and (c) $C^0$IP FEM estimator of Algorithm~\ref{algo_adaptive}.}
	\label{fig:conv_sing_adap}
\end{figure}

\begin{algorithm}\caption{Adaptive algorithm}\label{algo_adaptive}
	{\bf Input}: Initial mesh $\cT_0$, $J \geq 1$, bulk parameter $\theta \in (0,1]$.\\
    {\bf for} $j=0,1,2,\ldots$ {\bf do}\\
    Solve. Compute the discrete solution of Morley FEM $\Psi_j:=\Psi_{\M}$ or $C^0$IP $\Psi_j:=\Psi_{\ip}$ or DGFEM $\Psi_j:=\Psi_{\dg}$ on the mesh $\cT_j$.\\
    Estimate. For $K\in\cT_j$, compute the local contributions $\eta_{j}^2(K):=\eta_{\M}^2(K)$ of Morley FEM \cite{CCGMNN_Semilinear}, $\eta_{j}^2(K):=\eta_{\ip}^2(K)$ of $C^0$IP \cite{CCGMNN_DG} and $\eta_{j}^2(K):=\eta_{\dg}^2(K)$ of DGFEM \cite{CCGMNN_DG} defined by
    \begin{align*}
    \eta_{\M}^2(K):=& h_K^4\left\|[u_M,v_M]+f\right\|_{L^2(K)}^2+
    h_K^4 \left\|[u_M,u_M]\right\|_{L^2(K)}^2\\
   &+h_E\left\|\left[D^2 u_M\right]_E\tau_E\right\|_{L^2(E)}^2+h_E\left\|\left[D^2   v_M\right]_E\tau_E\right\|_{L^2(E)}^2,\\
    \eta_{\ip}^2(K):=&h_K^4\Big{(}\|f+[u_{\ip},v_{\ip}]\|_{L^2(K)}^2+\|[u_{\ip},u_{\ip}]\|_{L^2(K)}^2\Big{)}\\
    &+h_E\left(\|[D^2 u_{\ip}\,\nu_E]_E\cdot\nu_E\|_{L^2(E)}^2+\|[ D^2v_{\ip}\nu_E]_E\cdot\nu_E\|_{L^2(E)}^2\right)\\
    &\quad+h_E^{-1}\left(\|[\nabla u_{\ip}]_E\|_{L^2(E)}^2+\|[\nabla v_{\ip}]_E\|_{L^2(E)}^2\right),\\
    \eta_{\dg}^2(K):=&h_K^4\Big{(}\|f+[u_{\dg},v_{\dg}]\|_{L^2(K)}^2+\|[u_{\dg},u_{\dg}]\|_{L^2(K)}^2\Big{)}\\
    &+h_E^{-3}\left(\|[u_{\dg}]_E\|_{L^2(E)}^2+\|[v_{\dg}]_E\|_{L^2(E)}^2\right)+h_E^{-1}\left(\|[\nabla u_{\dg}]_E\|_{L^2(E)}^2+\|[\nabla v_{\dg}]_E\|_{L^2(E)}^2\right).
    \end{align*}
    Mark. The D\"{o}rfler marking chooses a minimal subset $\mathcal{M}_{j}\subset \cT_{j}$ such that
    \begin{equation*}
    \theta\sum_{K\in\mathcal{T}_{j}}\eta^2_{j}(K)\leq \sum_{K\in\mathcal{M}_{j}}\eta^2_{j}(K). 
    \end{equation*} 
    Refine. Compute the closure of $\mathcal{M}_{j}$ and generate a new mesh $\cT_{j+1}$ using newest vertex bisection \cite{Stevenson08}.\\
    {\bf Output}: Sequence of meshes $(\cT_j)_j$ and discrete solution $(\Psi_j)_j$.
\end{algorithm}	


\section*{Acknowledgements}
{ The author would like to thank Professor Neela Nataraj for fruitful discussions. The author acknowledges the support of National Board for Higher Mathematics (NBHM) research grant no. 0204/58/2018/R\&D-II/14746. 
}
\bibliography{refs}
\bibliographystyle{amsplain}

\bigskip

\end{document}